\documentclass[reqno,a4paper, 12pt]{amsart}

\usepackage[pdfpagelabels]{hyperref}
\usepackage{graphicx}

\usepackage[ansinew]{inputenc}
\usepackage{amsfonts,epsfig}
\usepackage{latexsym}
\usepackage{amsmath}
\usepackage{amssymb}

\newtheorem{theorem}{Theorem}
\newtheorem{lemma}[theorem]{Lemma}
\newtheorem{corollary}[theorem]{Corollary}
\newtheorem{proposition}[theorem]{Proposition}
\newtheorem{oldtheorem}{Theorem}

\theoremstyle{definition}

\theoremstyle{remark}

\numberwithin{equation}{section}

\newcommand{\set}[1]{\left\{#1\right\}}
\newcommand{\abs}[1]{\lvert#1\rvert}
\newcommand{\nm}[1]{\lVert#1\rVert}
\newcommand{\bnm}[1]{\big\Vert#1\big\Vert}

\newcommand{\D}{\mathbb{D}}

\newcommand{\N}{\mathbb{N}}

\newcommand{\C}{\mathbb{C}}

\newcommand{\RM}{\mathcal{R}}
\renewcommand{\phi}{\varphi}

\DeclareMathOperator{\diam}{Diam}

\newcommand{\BMOA}{\rm BMOA}
\newcommand{\BLOCH}{\mathcal{B}}
\newcommand{\ULU}{\mathcal{U}}

                  \def\z{\zeta}
                  \def\vp{\varphi}

\begin{document}

\title[Mean growth and geometric zero distribution of solutions]
{Mean growth and geometric zero distribution of solutions of linear differential equations}
\thanks{The first author is supported by the Academy of Finland \#258125; the second author is supported in part by the grants
MTM2011-24606 and 2014SGR 75; and the third author is supported in part by the Academy of Finland \#268009, by the 
Faculty of Science and Forestry of University of Eastern Finland \#930349, and by the grant MTM2011-26538.}

\author{Janne Gr\"ohn}
\address{Department of Physics and Mathematics, University of Eastern Finland, P.O. Box 111, FI-80101 Joensuu, Finland}
\email{janne.grohn@uef.fi}

\author{Artur Nicolau}
\address{Departament de Matem\`atiques, Universitat Aut\`onoma de Barcelona, 08193, Bellaterra, Barcelona, Spain}
\email{artur@mat.uab.cat }

\author{Jouni R\"atty\"a}
\address{Department of Physics and Mathematics, University of Eastern Finland, P.O. Box 111, FI-80101 Joensuu, Finland}
\email{jouni.rattya@uef.fi}



\date{\today}


\begin{abstract}
The aim of this paper is to consider certain conditions on 
the coefficient $A$ of the differential equation $f''+Af=0$ in the unit disc, 
which place all normal solutions $f$ to the union of Hardy spaces or result in 
the zero-sequence of each non-trivial solution to be uniformly separated. The conditions 
on the coefficient are given in terms of Carleson measures.
\end{abstract}

\maketitle



\section{Introduction} \label{sec:intro}

We consider solutions of the linear differential equation
\begin{equation} \label{eq:de2}
f''+A f=0
\end{equation}
in the unit disc $\D$ of the complex plane $\C$.
Recall that, if $f_1$ and $f_2$ are linearly independent solutions of \eqref{eq:de2}, then
the Schwarzian derivative
\begin{equation*}
S_w = \left( \frac{w''}{w'} \right)' - \frac{1}{2} \left(\frac{w''}{w'} \right)^2
\end{equation*}
of the quotient $w=f_1/f_2$ satisfies $S_w=2A$. 
One of our main objectives is to explore conditions on the coefficient $A$ placing
all solutions of \eqref{eq:de2} to the union of Hardy spaces, while the other aim is to study
the geometric zero distribution of non-trivial solutions $f\not\equiv 0$ of \eqref{eq:de2}.
In other words, we study restrictions of the Schwarzian derivative $S_w$ of
a locally univalent meromorphic function $w=f_1/f_2$ that place
$1/w'$ (which reduces to a~constant multiple of $f_2^2$) to the union of Hardy spaces,
and consider the geometric distribution of complex $a$-points of $w$ (which are precisely
the zeros of the solution $f_1-a f_2$ if $a\in\C$, and the zeros of $f_2$ if $a=\infty$).

We begin by recalling some notation. 
For $0<p<\infty$, the Hardy space $H^p$ consists of those
analytic functions in $\D$ for which
 \begin{equation*}
 \nm{f}_{H^p} = \lim_{r\to 1^-} \left( \frac{1}{2\pi} \int_0^{2\pi}
 \abs{f(r e^{i\theta})}^p \, d\theta \right)^{1/p}<\infty.
 \end{equation*}
A positive Borel measure $\mu$ on $\D$ is called a Carleson measure, if there exists a~positive constant 
$C$ such that
\begin{equation*}
\int_{\D} |f(z)|^p \, d\mu(z) \leq C \, \nm{f}_{H^p}^p, \quad f\in H^p.
\end{equation*}
These measures were characterized by Carleson as those positive measures $\mu$ for which
there exists a positive constant $K$ such that $\mu(Q) \leq K \ell(Q)$ for any
Carleson square $Q\subset\D$, or equivalently, as those positive measures $\mu$ for which
\begin{equation*}
\sup_{a\in\D} \int_\D \frac{1-\abs{a}^2}{|1-\overline{a}z|^2} \, d\mu(z) < \infty.
\end{equation*}
For more information see \cite{D:2000,G:2007}, for example.

The following spaces of analytic functions seem to be natural for the coefficient.
For $0\leq \alpha<\infty$, 
the growth space $H^\infty_\alpha$ contains those analytic functions $A$
in $\D$ for which
\begin{equation*}
\nm{A}_{H^\infty_\alpha} = \sup_{z\in\D} \, (1-\abs{z}^2)^\alpha \abs{A(z)} < \infty.
\end{equation*} 
For $0<p<\infty$, the analytic function $A$ in~$\D$
is said to belong to the space~$F^p$ if
$|A(z)|^p(1-|z|^2)^{2p-1}\, dm(z)$ is a Carleson measure, and we denote
\begin{equation*}
\nm{A}_{F^p} = \left( \, \sup_{a\in\D} \, \int_{\D} \abs{A(z)}^p
(1-\abs{z}^2)^{2p-2} ( 1 - \abs{\varphi_a(z)}^2)
     \, d m(z) \right)^{1/p}.
\end{equation*}
Here $\varphi_a(z) = (a-z)/(1-\overline{a} z)$ and $d m(z)$
denotes the element of the Lebesgue area measure on $\D$.
Note that $F^p\subsetneq H^\infty_2$ for any $0<p<\infty$ by subharmonicity. 
Correspondingly, the ``little-oh'' space
$F^p_0$ (the closure of polynomials in $F^p$), which
consists of those analytic functions in $\D$ for which
$|A(z)|^p(1-|z|^2)^{2p-1}\, dm(z)$ is a vanishing Carleson measure,
and the space $H^\infty_{\alpha,0}$, which contains those analytic functions
in $\D$ for which
\begin{equation*}
\lim_{|z|\to 1^-} \, (1-\abs{z}^2)^\alpha \abs{A(z)} = 0,
\end{equation*}
satisfy the inclusion $F^p_0\subsetneq H^\infty_{2,0}$ for any $0<p<\infty$.
It is known that:
\begin{itemize}
\item[\rm (i)] for each $0<p<\infty$ there exists a positive constant $\alpha = \alpha(p)$ 
such that, if $\nm{A}_{F^2} \leq \alpha$, then all solutions
$f$ of \eqref{eq:de2} belong to $H^p$, and each
non-trivial solution $f$ has at most one zero in $\D$;

\item[\rm (ii)] if $A\in F^2_0$, then all solutions $f$ of
\eqref{eq:de2} belong to $\bigcap_{0<p<\infty} H^p$, and each
non-trivial solution $f$ has at most finitely many zeros in $\D$.

\end{itemize}
The assertion (i) follows from
\cite[Theorem~1.7]{R:2007}; note that by choosing a sufficiently
small $0<\alpha<\infty$
we have $\nm{A}_{H^\infty_2}\leq 1$, and hence each solution $f$ of
\eqref{eq:de2} vanishes at most once in $\D$
\cite[Theorem~1]{N:1949}. The case (ii) 
is a consequence of \cite[p.~789]{R:2007}; in this case $A\in
H^\infty_{2,0}$, and hence all solutions have at most finitely
many zeros in $\D$ \cite[Theorem~1]{S:1955}. 
Notice also that
each non-trivial solution of \eqref{eq:de2} has at most finitely many zeros provided that $A\in F^2$ is lacunary,
since the lacunary series in $F^2$ and $F^2_{0}$ are same.


\section{Results} \label{subsec:uni}

Our results concern interrelationships of the properties
\begin{enumerate}
\item[\rm (i)]
all solutions of \eqref{eq:de2} belong to the union of Hardy spaces;

\item[\rm (ii)]
the zero-sequence of each non-trivial solution of 
\eqref{eq:de2} is uniformly separated;

\item[\rm (iii)]
the growth of the coefficient $A$.
\end{enumerate}

In the light of the following results it seems plausible that
$|A(z)|^2(1-|z|^2)^3 \, dm(z)$ being a~Carleson measure is sufficient for (i) and (ii).
The following theorems give many partial results in this direction.
For example, the existence of 
one non-vanishing solution allows us to reach this conclusion, see Corollary~\ref{cor:33}.

We begin with the geometric zero distribution of solutions of \eqref{eq:de2}.
Recall that the sequence $\{z_n\}_{n=1}^\infty\subset \D$
is called uniformly separated if
\begin{equation*}
\inf_{k\in\N} \, \prod_{n\in\N\setminus\{k\}} \left| \frac{z_n - z_k}{1-\overline{z}_n z_k} \right|>0,
\end{equation*}
while $\{z_n\}_{n=1}^\infty\subset \D$ is said to be separated in the hyperbolic metric if
there exists a~constant $\delta=\delta(\{z_n\}_{n=1}^\infty)>0$ such that
$\varrho_p(z_n,z_k) = |z_n-z_k|/|1-\overline{z}_n z_k|>\delta$ for all $n,k\in\N$ for which $n\neq k$.


\begin{theorem} \label{thmnew:2}
If $|A(z)|(1-|z|^2)\, dm(z)$ is a Carleson measure, 
then the zero-sequence of each non-trivial solution $f$ of \eqref{eq:de2} is uniformly separated.
\end{theorem}

Theorem~\ref{thmnew:2} should be compared with \cite[Theorem~3]{S:1955} which states that,
if $A\in H^\infty_2$, then the zero-sequence of each non-trivial solution $f$ of \eqref{eq:de2} is separated in the hyperbolic
metric by a constant depending on $\nm{A}_{H^\infty_2}$. For the interplay between the maximal
growth of the coefficient $A$ and the minimal separation of the zeros of non-trivial solutions $f$ of \eqref{eq:de2},
we refer to \cite{CGHR:2013}.

The proof of Theorem~\ref{thmnew:2} relies on Theorem~\ref{th:jh} below, 
according to which all solutions $f$ of \eqref{eq:de2} belong to the Nevanlinna class $N$, that is
\begin{equation} \label{eq:prox}
\sup_{0\leq r < 1}  m(r,f) = \sup_{0\leq r < 1} \, \frac{1}{2\pi} \int_0^{2\pi} \log^+ |f(re^{i\theta})| \, d\theta < \infty,
\end{equation}
if
\begin{equation} \label{eq:Acond}
\int_{\D} |A(z)| (1-|z|^2) \, dm(z)<\infty.
\end{equation}
Condition \eqref{eq:Acond} is, of course, satisfied by the hypothesis of Theorem~\ref{thmnew:2}.


\subsection{Differential equations with one non-vanishing solution} \label{sec:zerofree}

The following result based on \cite[Corollary~7]{BJ:1994} introduces a factorization of solutions of \eqref{eq:de2},
which does not have an~apparent counterpart in general.


\begin{theorem} \label{thm:fact}
Let $A$ be an analytic function in $\D$, and suppose that \eqref{eq:de2} admits
a~non-vanishing solution $g$.

\begin{enumerate}
\item[\rm (i)]
If $|A(z)|^2(1-|z|^2)^3 \, dm(z)$ is a Carleson measure, then 
all non-trivial solutions $f$ of \eqref{eq:de2} can be factorized as $f=g W$, where
$\log g \in\BMOA$, and
either $\log W'\in\BMOA$ (if $f$ is linearly independent
to $g$) or $W$ is a non-zero complex constant (if $f$ is linearly dependent to $g$).

\item[\rm (ii)]
If $A\in H^\infty_2$, then 
all non-trivial solutions $f$ of \eqref{eq:de2} can be factorized as $f=g W$, where
$\log g \in\BLOCH$, and
either $\log W'\in\BLOCH$ (if $f$ is linearly independent
to $g$) or $W$ is a non-zero complex constant (if $f$ is linearly dependent to $g$).
\end{enumerate}
\end{theorem}

Recall that $\BMOA$ consists those $f\in H^2$ for which
$|f'(z)|^2(1-|z|^2) \, dm(z)$ is a~Carleson measure, and $\BMOA$ has the~seminorm
\begin{equation*}
 \nm{f}_{\BMOA} =  \left( \sup_{a\in\D} \int_\D |f'(z)|^2 ( 1 - |\vp_a(z)|^2) \, dm(z)<\infty \right)^{1/2}.
 \end{equation*}
The Bloch space $\BLOCH$
contains those analytic functions $f$ in $\D$ for which $f'\in H^\infty_1$.


\begin{corollary} \label{cor:33}
Let $A$ be an~analytic function in $\D$, and suppose that \eqref{eq:de2} admits
a~non-vanishing solution.
If $|A(z)|^2(1-|z|^2)^3 \, dm(z)$ is a Carleson measure, 
then all solutions of \eqref{eq:de2} belong to $H^p$ for some $0<p<\infty$, and
the zero-sequence of each non-trivial solution $f$ of \eqref{eq:de2} is uniformly separated.
\end{corollary}

The statement corresponding to Corollary~\ref{cor:33} in the case of $A\in H^\infty_2$ is true
without the assumption of existence of a non-vanishing solution:
if $A\in H^\infty_2$, 
then all solutions of \eqref{eq:de2} belong to $H^\infty_p$ for some $p=p(\nm{A}_{H^\infty_2})>0$
\cite[Example~1]{P:1982}, and
the zero-sequence of each non-trivial solution $f$ of \eqref{eq:de2} is separated in the hyperbolic
metric by a~positive constant depending on $\nm{A}_{H^\infty_2}$ \cite[Theorem~3]{S:1955}.

The following result gives a complete description of the zero-free solutions
in our setting.


\begin{theorem} \label{thm:cara}
Let $A$ be an analytic function in $\D$.
\begin{enumerate}
\item[\rm (i)]
If $|A(z)|^2(1-|z|^2)^3\, dm(z)$ is a~Carleson measure, then all non-vanishing solutions $f$
of \eqref{eq:de2} satisfy $\log f\in\BMOA$. Conversely, if \eqref{eq:de2} admits
a zero-free solution $f$ satisfying $\log f\in\BMOA$, then $|A(z)|^2(1-|z|^2)^3\, dm(z)$ 
is a~Carleson measure.

\item[\rm (ii)]
If $A\in H^\infty_2$, then all non-vanishing solutions $f$
of \eqref{eq:de2} satisfy $\log f\in\BLOCH$. Conversely, if \eqref{eq:de2} admits
a zero-free solution $f$ satisfying $\log f\in\BLOCH$, then $A\in H^\infty_2$.
\end{enumerate}
\end{theorem}

In the case of Theorem~\ref{thm:cara}(i) all non-vanishing
solutions $f$ of \eqref{eq:de2} are in fact outer functions in \emph{Hardy spaces}, see \cite[Corollary~3, p.~34]{D:2000}.
Note that this property restricts not only the growth of non-vanishing solutions but also the rate at which
they may decay to zero.
Correspondingly, if $f$ is a zero-free solution of \eqref{eq:de2} with $A\in H^\infty_2$,
then Theorem~\ref{thm:cara}(ii) implies that there exists a~constant $p=p(\nm{A}_{H^\infty_2})$ with $0<p<\infty$ such that
$(1-|z|)^p\lesssim|f(z)|\lesssim (1-|z|)^{-p}$ for all $z\in\D$.
We employ the notation $a\asymp b$, which is equivalent to the
conditions $a\lesssim b$ and $b\lesssim a$, where the former means
that there exists a constant $C>0$ such that $a \leq C b$, and the
latter is defined analogously.

For example, the argument above asserts that 
the singular inner function $$f(z) = \exp\big(-(1+z)/(1-z)\big), \quad z\in\D,$$
cannot be a solution of \eqref{eq:de2} with $A\in H^\infty_2$.
This is also easily verified by a direct computation, since in this case 
$A(z)=-4z(1-z)^{-4}$ by \eqref{eq:de2}. In particular, the boundedness
of one solution of \eqref{eq:de2} is not sufficient to guarantee that $A\in H^\infty_2$.

One of the main tools concerning the results in Section~\ref{sec:zerofree} is 
\cite[Corollary~7]{BJ:1994}, which
also induces a~growth estimate for the non-vanishing solutions of \eqref{eq:de2}.


\begin{proposition} \label{prop:bjest}
Let $A$ be analytic in $\D$.
If $f$ is a non-vanishing solution of \eqref{eq:de2}, then
\begin{equation*}
\frac{1}{2\pi} \int_0^{2\pi} \left| \log \frac{f(re^{i\theta})}{f(0)} \right|^2  d\theta
  \lesssim r^2 \left| \frac{f'(0)}{f(0)} \right|^2 + r^2 \int_{D(0,r)} |A(z)|^2 (1-|z|^2)^3 \, dm(z)
\end{equation*}
for all $0<r<1$.
\end{proposition}

By Proposition~\ref{prop:bjest} and \cite[Corollary~3, p.~34]{D:2000} all non-vanishing
solutions of \eqref{eq:de2} are outer functions in the \emph{Nevanlinna class} provided that
\begin{equation} \label{eq:conj}
\int_{\D} |A(z)|^2 (1-|z|^2)^3 \, dm(z)<\infty.
\end{equation}
It would be desirable to show that \eqref{eq:conj} (we may also assume that $A\in H^\infty_2$)
is sufficient to place all solutions of \eqref{eq:de2} in $N$. This would 
improve the results in the
literature (under the additional condition $A\in H^\infty_2$), since
\begin{equation} \label{eq:compgrests}
    \begin{split}
    \int_{\D} |A(z)|^2 (1-|z|^2)^3\, dm(z)
    &\le\|A\|_{H^\infty_2} \int_{\D} |A(z)|(1-|z|^2)\, dm(z)\\
    &\le \|A\|_{H^\infty_2}^{3/2}  \int_{\D} |A(z)|^{1/2}\, dm(z).
    \end{split}
    \end{equation}
The last integral in \eqref{eq:compgrests} appears in the growth estimate \cite[Theorem~5]{P:1982}, which is obtained by
Herold's comparison theorem, while the intermediate integral shows up in
Theorem~\ref{th:jh}, which is proved by integrating \eqref{eq:de2}
and applying the Gronwall lemma. 

We close our discussion on non-vanishing solutions by the
following result, which shows that there are differential equations \eqref{eq:de2} in $\D$
having no zero-free solutions. The proof of Theorem~\ref{thm:roth} was constructed jointly with
Professor O.~Roth.


\begin{theorem}\label{thm:roth}
There exists a locally univalent meromorphic function in $\D$, which maps $\D$
onto the extended complex plane.
\end{theorem}


\subsection{Normal solutions of differential equations}

By Corollary~\ref{cor:33}, the existence of one non-vanishing solution of \eqref{eq:de2} allows us
to survey many specific properties of all solutions of \eqref{eq:de2}.
We now drop the additional assumption on the existence of a non-vanishing solution, and 
proceed to study solutions which may have zeros, but
whose behavior around their zeros is in a~certain sense regular.
This leads us to the concept of normality.
Recall that the meromorphic function $f$ in $\D$ is called normal (in the sense of Lehto and Virtanen) if and only if
\begin{equation} \label{eq:normaldef}
\sigma(f) = \sup_{z\in\D} \, (1-|z|^2) \frac{|f'(z)|}{1+|f(z)|^2} < \infty;
\end{equation}
for more information on normal functions, see for example \cite{LV:1957}.


\begin{proposition} \label{prop:normal}
Let $f$ be a non-trivial solution of \eqref{eq:de2} with $A\in H^\infty_2$, and let $\{z_n\}_{n=1}^\infty$ be the zero-sequence of $f$. 
Then, the following conditions are equivalent:
\begin{enumerate}
\item[\rm (i)]
$f$ is normal;

\item[\rm (ii)]
$\sup_{n\in\N} \, (1-|z_n|^2)|f'(z_n)|<\infty$;

\item[\rm (iii)]
$f$ is uniformly bounded in $\bigcup_{n=1}^\infty D\big( z_n,c(1-|z_n|) \big)$ for some $0<c<1$.
\end{enumerate}
\end{proposition}

By Proposition~\ref{prop:normal} every solution of \eqref{eq:de2} with $A\in H^\infty_2$, which has only finitely
many zeros, is normal. In particular, if $A\in H^\infty_{2,0}$, then 
all solutions of \eqref{eq:de2} 
are normal \cite[Theorem~1]{S:1955} --- yet all non-trivial solutions may lie outside $N$
\cite[pp.~57-58]{H:2013}. 

Corollary~\ref{cor:34} below states that
all normal solutions of \eqref{eq:de2} belong to certain Hardy space,
under the assumption that $|A(z)|^2(1-|z|^2)^3\, dm(z)$ is a~Carleson measure.
The proof of Corollary~\ref{cor:34} is based on the following result,
whose proof bears similarity to that of \cite[pp.~105-107]{BJ:1994}.


\begin{theorem}\label{thm:carleson}
If $w$ is a meromorphic function in $\D$
such that $|S_w(z)|^2 (1-|z|^2)^3\, dm(z)$ is a~Carleson measure and $w'$ is normal, 
then for all sufficiently small $0<p<\infty$ 
there exists a constant $C=C(p,\nm{S_w}_{F^2}, \sigma(w'))$
with $1<C<\infty$ such that $\nm{1/w'}_{H^p} \leq C$.
\end{theorem}

If we assume in Theorem~\ref{thm:carleson} that $w$ has only finitely many poles
(which all are simple, since the analyticity of $S_w$ implies that $w$ is locally univalent), 
then the assumption on the normality of $w'$ is not needed. In fact, this follows from the other assumptions,
since in this case $f = 1/\sqrt{w'}$ is a solution of $f''+(1/2) S_w f = 0$ having only finitely many zeros, 
and hence $f$ is normal by Proposition~\ref{prop:normal}. As a consequence we deduce that $w'$ is normal.


\begin{corollary} \label{cor:34}
If $|A(z)|^2(1-|z|^2)^3\, dm(z)$ is a~Carleson measure,
then all normal solutions $f$ of \eqref{eq:de2} belong to $H^p$ for any sufficiently small $0<p<\infty$.
\end{corollary}


\section{Proof of Theorem~\ref{thmnew:2}}

The following result concerns
the growth of solutions of linear differential equations.
Recall that $\log^+ x = \max\{ \log x, 0 \}$ and 
$m(r,g)$ is the Nevanlinna proximity function of~$g$, as in \eqref{eq:prox}.


\begin{oldtheorem}[\protect{\cite[Theorem~4.5]{H:2000}}] \label{th:jh}
If
\begin{equation*} 
\int_{\D} |B(\zeta)| (1-|\zeta|^2) \, dm(\zeta) < \infty,
\end{equation*}
then every solution $g$ of $g''+B g = 0$ is of bounded
characteristic, and 
\begin{equation*}
m(r,g)
     \leq \log^+ \big( |g(0)| + |g'(0)| \big)
         + K \int_{D(0,r)} |B(\zeta)| (1-|\zeta|^2) \, dm(\zeta), \quad 0\leq r < 1,
\end{equation*}
where $0<K<\infty$ is an absolute constant.
\end{oldtheorem}

We proceed to prove Theorem~\ref{thmnew:2}. Let $\kappa\in\D$. If $f$ is a solution of \eqref{eq:de2}, then
\begin{equation*}
g_\kappa(\zeta) = \gamma \, f\big( \varphi_\kappa(\zeta)\big) \big( \varphi_\kappa'(\zeta) \big)^{-1/2}, \quad \gamma\in\C,
\end{equation*}
is a solution of
\begin{equation} \label{eq:transde2}
g_\kappa'' + B_\kappa g_\kappa = 0, \quad B_\kappa(\zeta) =
A\big( \varphi_\kappa(\zeta) \big)
              \varphi_\kappa'(\zeta)^2 + \frac{1}{2} S_{\varphi_\kappa}(\zeta), \quad \zeta\in\D,
\end{equation}
see for example \cite[p.~394]{I:1944} or \cite[Lemma~1]{P:1982}.
Here $\varphi_\kappa(\zeta) =(\kappa-\zeta)/(1-\overline{\kappa}\zeta)$, and hence
the Schwarzian derivative $S_{\varphi_\kappa}$ vanishes identically.
The change of variable $z=\varphi_\kappa(\zeta)$ implies
\begin{equation*}
\sup_{\kappa\in\D} \, \int_{\D} |B_\kappa(\zeta)| (1-|\zeta|^2) \,
dm(\zeta)
        = \sup_{\kappa\in\D} \, \int_{\D} |A(z)| \big( 1 - |\varphi_\kappa(z)|^2 \big) \, dm(z)
           = \nm{A}_{F^1},
\end{equation*}
and by means of Theorem~\ref{th:jh} we obtain
\begin{equation} \label{eq:trv2}
m(r,g_\kappa) 
\leq \log^+ \big( |g_\kappa(0)| + |g'_\kappa(0)| \big) + K  \, \nm{A}_{F^1}, \quad 0\leq r < 1, \quad \kappa\in\D,
\end{equation}
for some absolute constant $0<K<\infty$.

Let $\{ z_n\}_{n=1}^\infty$ be the zero-sequence of $f$.
For each $k\in\N$, we have $g_{z_k}(0)=0$, and this zero of $g_{z_k}$ at the origin
is simple as all the zeros of all non-trivial solutions of \eqref{eq:transde2} are.
Since $g'_{z_k}(0)\neq 0$,  we may normalize the solution $g_{z_k}$
to satisfy
\begin{equation*}
g_{z_k}'(0)= \gamma \, f'(z_k) (|z_k|^2-1)^{1/2} = 1
\end{equation*}
by choosing the constant $\gamma=\gamma(f,k)$ appropriately.
We proceed to prove that the Blaschke sum regarding the
zeros of the normalized solution  $g_{z_k}$ of
\eqref{eq:transde2} is uniformly bounded for all $k\in\N$.
Without loss of generality, we may suppose that the zeros
$\{\zeta_{n,k} \}_{n=1}^\infty$ of $g_{z_k}$ satisfy
$|\zeta_{1,k}| = 0 < |\zeta_{2,k}| \leq |\zeta_{3,k}| \leq \cdots$.
By applying Jensen's formula to $z\mapsto z^{-1} g_{z_k}(z)$ results in
\begin{equation*} 
\frac{1}{2\pi} \int_0^{2\pi} \log{ \big|
g_{z_k}(re^{i\theta})\big|} \, d\theta
           = \sum_{\substack{n\in\N\\0<|\zeta_{n,k}|<r}} \log\frac{r}{|\zeta_{n,k}|} + \log{r}, \quad 0< r < 1.
\end{equation*}
Letting $r\to 1^-$, and taking account on \eqref{eq:trv2}, we get
\begin{equation} \label{eq:upest}
\sup_{k\in\N} \, \sum_{n=2}^\infty \big( 1-|\zeta_{n,k}| \big)
\leq  K  \, \nm{A}_{F^1}.
\end{equation}
Since the zeros of $g_{z_k}$ are precisely the images of the
zeros of $f$ under the mapping~$\varphi_{z_k}$,
\eqref{eq:upest} implies that the zero-sequence $\{z_n\}_{n=1}^\infty$ of $f$
satisfies
\begin{equation} \label{eq:unisep}
\begin{split}
\sup_{k\in\N} \, \sum_{n\neq k} \left( 1 - \left|
\frac{z_n-z_k}{1-\overline{z}_n z_k} \right| \right)
       = \sup_{k\in\N} \, \sum_{n=2}^\infty \big( 1-|\zeta_{n,k}| \big)
     \leq  K  \, \nm{A}_{F^1}.
\end{split}
\end{equation}

Since $A\in H^\infty_2$, $\{z_n\}_{n=1}^\infty$ is separated
\cite[Theorem~3]{S:1955}. Let $0<\delta<1$ be a~constant such that
$\varrho_p(z_n,z_k)\geq \delta$ for all natural numbers $n\neq k$, where
$\varrho_p$ stands for the pseudo-hyperbolic distance. Now
\eqref{eq:unisep}, and the inequality $-\log{x} \leq x^{-1} (1-x)$ for $0<x<1$, imply 
\begin{equation*}
\sup_{k\in\N} \, \sum_{n\neq k} -\log \left|
\frac{z_n-z_k}{1-\overline{z}_n z_k} \right|
       \leq \frac{1}{\delta} \, \sup_{k\in\N} \, \sum_{n\neq k}  \left( 1 - \left| \frac{z_n-z_k}{1-\overline{z}_n z_k} \right| \right)
      \leq \frac{ K  \, \nm{A}_{F^1}}{\delta}.
\end{equation*}
We conclude that $\{z_n\}_{n=1}^\infty$ is uniformly separated.


\section{Proof of Theorem~\ref{thm:fact}} \label{sec:fact}

(i) Let $\{f_1,f_2\}$ be a solution base of \eqref{eq:de2} such that $g=f_2$ is non-vanishing, and the Wronskian
determinant $W(f_1,f_2)=-1$. Then $w=f_1/f_2$ is analytic, locally univalent, and it satisfies 
$S_w=2A$ and $w'=f_2^{-2}$.
Define  $h_\z(z) = \log w'(\vp_\z(z))$, where $\vp_\z(z)=(\z-z)/(1-\overline{\z}z)$ and $\z\in\D$.
According to \cite[Corollary~7]{BJ:1994}, 
\begin{equation*}
\bnm{h_\z-h_\z(0)}^2_{H^2} \lesssim \abs{h_\z'(0)}^2 + \int_{\D} \left| h_\z''(z)-h_\z'(z)^2/2 \right|^2 (1-\abs{z}^2)^3 \, dm(z). 
\end{equation*}
By a direct computation,
\begin{equation*}
h'_\z(z) = \frac{w''\big(\vp_\z(z)\big) \vp_\z'(z)}{w'\big( \vp_\z (z) \big)}, \quad
h_\z''(z) -\frac{h_\z'(z)^2}{2} = S_w\big( \vp_\z(z) \big) \big( \vp_\z'(z) \big)^2+ \frac{w''\big(\vp_\z(z)\big) \vp_\z''(z)}{w'\big( \vp_\z (z) \big)},
\end{equation*}
which implies
\begin{align}
\sup_{\z\in\D} \, \bnm{h_\z-h_\z(0)}^2_{H^2}
     & \lesssim \, \bnm{w''/w'}^2_{H^\infty_1}  \label{eq:j1}\\
       & \quad       +  \sup_{\z\in\D} \, \int_{\D} \left| S_w\big( \vp_\z(z) \big) \right|^2 \big| \vp_\z'(z) \big|^4 (1-\abs{z}^2)^3 \, dm(z) \label{eq:j2}\\
       & \quad       +  \sup_{\z\in\D} \, \int_{\D} \left|\frac{w''\big(\vp_\z(z)\big)}{w'\big( \vp_\z (z) \big)} \right|^2 
       \left| \vp_\z''(z) \right|^2(1-\abs{z}^2)^3 \, dm(z) \label{eq:j3}.
\end{align}
The right-hand side of \eqref{eq:j1} is finite, since $S_w\in H^\infty_2$ \cite[Theorem~2]{Y:1977}, while
\eqref{eq:j2} reduces to $\nm{S_w}^2_{F^2}$. 
Furthermore, \eqref{eq:j3} is finite, since it is bounded by
\begin{equation*}
 \bnm{w''/w'}_{H^\infty_1}^2 \, \sup_{\z\in\D} \, \int_{\D} 
       \bigg| \frac{\vp_\z''(z)}{\vp_\z'(z)} \bigg|^2(1-\abs{z}^2) \, dm(z) < \infty.
\end{equation*}
We conclude that $\log w'$ belongs to $\BMOA$. 

Now, any solution $f\not\equiv 0$ of \eqref{eq:de2}, which is linearly independent to $f_2$,
can be written as $f = \alpha f_1 + \beta f_2 = (\alpha w + \beta) f_2$, where $\alpha\neq 0$.
Then, functions $W= \alpha w + \beta$ and $g=f_2$ satisfy the assertion, since
$\log W' = \log\alpha + \log w'  \in\BMOA$ and $\log g = -2 \log w' \in \BMOA$
by the argument above. Moreover, all solutions $f$ of \eqref{eq:de2}, which are linearly dependent to $f_2$, 
satisfy $\log f = \log\beta + \log f_2 \in \BMOA$. The assertion (i) is proved.

(ii) Let $\{f_1,f_2\}$ be a solution base of \eqref{eq:de2} such that $g=f_2$ is non-vanishing, and the Wronskian
determinant $W(f_1,f_2)=-1$. Then $w=f_1/f_2$ is analytic, locally univalent, and it satisfies 
$S_w=2A$ and $w'=f_2^{-2}$. By the assumption $S_w\in H^\infty_2$, which implies
$\log w'\in\BLOCH$ \cite[Theorem~2]{Y:1977}. The assertion of (ii) follows as above.


\section{Proof of Corollary~\ref{cor:33}}

The following auxiliary result, which is well-known by experts, 
is proved for the convenience of the reader.


\begin{lemma}\label{lemma:prop1}
If $W$ is a locally univalent analytic function in $\D$ such that
$\log{W'}\in \BMOA$, then all finite $a$-points of $W$ (i.e.~solutions
of $W(z)=a\in\C$) are uniformly separated.
\end{lemma}
\begin{proof}
It suffices to prove the assertion for the zeros, for otherwise we may
consider the zeros of $W(z)-a$ for $a\in\C$. 
Let $\{z_n\}_{n=1}^\infty$ be the zero-sequence of $W$, and define
    $$
    h_{z_n}(z)=-\frac{W(\vp_{z_n}(z))}{W'(z_n)(1-|z_n|^2)},\quad z\in\D, \quad n\in\N.
    $$
Then $h_{z_n}(0)=0$ and $h_{z_n}'(0)=1$. Now $\log h'_{z_n}$ is well-defined and analytic, as $W$ is locally
univalent.
By the Littlewood-Paley identity \cite[Lemma~3.1, p.~228]{G:2007}, 
    \begin{equation*}
    \begin{split}
    \|\log h'_{z_n}\|_{\BMOA}^2
    &\asymp \sup_{a\in\D} \, \int_\D
    \left|\frac{W''(\vp_{z_n}(z))}{W'(\vp_{z_n}(z))}\vp'_{z_n}(z)+\frac{\vp_{z_n}''(z)}{\vp_{z_n}'(z)}\right|^2(1-|\vp_a(z)|^2)\,dm(z)\\
    &\lesssim \sup_{a\in\D}\,\int_\D
    \left|\frac{W''(\zeta)}{W'(\zeta)}\right|^2\big(1-\big|\vp_a(\vp_{z_n}(\zeta))\big|^2\big)\,dm(\zeta)\\
    &\qquad+\sup_{a\in\D}\,\int_\D
    \left|\frac{\vp_{z_n}''(z)}{\vp_{z_n}'(z)}\right|^2(1-|\vp_a(z)|^2)\,dm(z)\\
    &\lesssim\|\log W'\|_{\BMOA}^2+1, \quad n\in\N.
    \end{split}
    \end{equation*}
According to \cite[Theorem~1]{CS:1976}  we have $\sup_{n\in\N} \|h'_{z_n}\|_{H^p}\lesssim 1$ for sufficiently small
$p=p(\|\log W'\|_{\BMOA})$ with $0<p<\infty$, and hence $\sup_{n\in\N} \|h_{z_n}\|_{H^p}\lesssim 1$
by \cite[Theorem~5.12]{D:2000}. We deduce, by means of $\log x \leq p^{-1}  x^p$ for $0<x<\infty$, that
\begin{equation*}
\sup_{n\in\N} \, \frac{1}{2\pi} \int_0^{2\pi} \log | h_{z_n}(re^{i\theta}) | \, d\theta
  \lesssim \sup_{n\in\N} \, \|h_{z_n}\|^p_{H^p}\lesssim 1.
\end{equation*}
By applying Jensen formula to $z\mapsto z^{-1} h_{z_n}(z)$, and arguing as in the proof of
Theorem~\ref{thmnew:2}, we obtain the assertion.
\end{proof}

We proceed to prove Corollary~\ref{cor:33}.
Suppose that \eqref{eq:de2} admits a non-vanishing solution $g$, and let
$|A(z)|^2(1-|z|^2)^3\, dm(z)$ be a~Carleson measure. By Theorem~\ref{thm:fact}(i)
every non-trivial solution $f$ of \eqref{eq:de2} can be represented as $f=gW$, where
$\log g\in \BMOA$, and either $W$ is locally univalent such that $\log{W'}\in\BMOA$ or $W\in\C\setminus\{0\}$.

If $\log W'\in\BMOA$, 
then $W' = \exp( \log W' )$ belongs to some Hardy space by \cite[Theorem~1]{CS:1976}. 
This implies that also $W$ is in some Hardy space \cite[Theorem~5.12]{D:2000}. 
Analogously, $\log g\in\BMOA$ implies that $g=\exp(\log g)$ 
belongs to some Hardy space. In conclusion, under the assumptions
of Corollary~\ref{cor:33} every solution of \eqref{eq:de2} can be represented as a product
of two Hardy functions and hence every solution of \eqref{eq:de2} belongs to certain (fixed) Hardy
space.

Every non-trivial solution of \eqref{eq:de2}, which is linearly dependent to $g$, is zero-free, while
the zero-sequence of every non-trivial solution of \eqref{eq:de2}, which is linearly independent to $g$, 
is uniformly separated by Lemma~\ref{lemma:prop1}.


\section{Proof of Theorem~\ref{thm:cara}}

(i) The first part of the assertion follows directly from the proof of Theorem~\ref{thm:fact}(i),
since $g=f_2$ can be chosen to be any non-vanishing solution of \eqref{eq:de2}.

Conversely, suppose that \eqref{eq:de2} possesses a zero-free solution $f$ such that $\log f\in\BMOA$.
Let $\{f,g\}$ be a solution base of \eqref{eq:de2} such that $W(f,g)=1$, and let $w=g/f$.
Now $w'$ is locally univalent analytic function such that $\log w' = \log 1/f^2\in\BMOA$.
Moreover, the Schwarzian derivative $S_w=2A$ is analytic in $\D$. 
Since
\begin{equation*}
|S_w(z)|^2 \lesssim \left| \bigg( \frac{w''}{w'} \bigg)'(z)
\right|^2
         + \left| \frac{w''(z)}{w'(z)} \right|^4, \quad z\in\D,
\end{equation*}
and $\log w'\in\BMOA\subset\BLOCH$, standard estimates show that
$|S_w(z)|^2(1-|z|^2)^3\, dm(z)$ is a~Carleson measure.
The assertion of (i) follows.

(ii) Let $f$ be a zero-free solution of \eqref{eq:de2} with $A\in H^\infty_2$, and
let $\{f,g\}$ be a solution base such that
$W(f,g)=1$. Define $w=g/f$. Then $w$ is a locally
univalent analytic function satisfying $S_w=2A\in H^\infty_2$. 
By \cite[Theorem~2]{Y:1977} the pre-Schwarzian derivative
$w''/w' = (\log w')' \in H^{\infty}_1$, and hence $\log f = -2^{-1} \log w' \in\BLOCH$.

Conversely, suppose that \eqref{eq:de2} possesses a zero-free solution $f$ such that $\log f\in\BLOCH$.
Now, $A = - f''/f = - \left( f'/f \right)' - \left( f'/f \right)^2\in H^\infty_2$,
which concludes the proof of Theorem~\ref{thm:cara}.


\section{Proof of Proposition~\ref{prop:bjest}}

Let $f$ be a non-vanishing solution of \eqref{eq:de2}, and let $\{ f,g\}$ be a 
solution base of \eqref{eq:de2} such that $W(f,g)=1$. Define $w=g/f$, and notice that $w'$
is a locally univalent analytic function such that $w'=1/f^2$.
An application of \cite[Corollary~7]{BJ:1994} with $\varphi(z) = \log\big( w'(r z) r\big)$ yields
\begin{equation*}
\frac{1}{2\pi} \int_0^{2\pi} \bigg| \log\frac{w'(re^{i\theta})}{w'(0)} \bigg|^2 \, d\theta
\lesssim r^2 \left| \frac{w''(0)}{w'(0)} \right|^2
+ r^2 \int_{\D} |S_w(rz)|^2 (1-|z|^2)^3 r^2\, dm(z)
\end{equation*}
for $0<r<1$. This implies
\begin{equation*}
 \frac{4}{2\pi} \int_0^{2\pi} \bigg| \log\frac{f(re^{i\theta})}{f(0)} \bigg|^2 \, d\theta
    \lesssim 4r^2 \left| \frac{f'(0)}{f(0)} \right|^2
    + r^2 \int_{D(0,r)} |S_w(\zeta)|^2 \left(1-\frac{|\zeta|^2}{r^2} \right)^3 dm(\zeta)
\end{equation*}
for $0<r<1$, which proves the assertion.


\section{Proof of Theorem~\ref{thm:roth}}

Let $\widehat{\C} = \C \cup \set{\infty}$ denote the extended complex plane. Define
$R\colon \widehat{\C} \to \widehat{\C}$ by $R(z)=z+1/(2z^2)$.
Then $R$ is a rational function of degree three, and $R$ is locally univalent
in $\Omega = \widehat{\C}\setminus\set{\z_1,\z_2,\z_3,0}$, where $\z_1$,$\z_2$ and $\z_3$
are the three cube roots of unity.

Note that $R(z)=w$ has a solution $z\in\Omega$ for any $w\in\widehat{\C}$. First,
the function $R$ takes the value $w=\infty$ at the point $z=\infty\in\Omega$. Second, 
if $w\in\C$, then $z=0$ cannot be a solution of $R(z)=w$. Hence $R(z)=w$ is equivalent to
$2z^3-2 w z^2+1=0$, which has a solution $z\in\Omega$ for any $w\in\C$.
 
Let $M$ be the inversion $M(z)=1/z$, and  let $M^{-1}(\Omega)$ be the pre-image set of~$\Omega$.
Now, define $\Pi\colon \D \to M^{-1}(\Omega)$ to be a universal covering map.
Since $M^{-1}(\Omega)$ is a~plane set whose complement in $\C$ contains three points,
we may assume that $\Pi$ is analytic \cite[p.~125, Theorem~5.1]{C:1995}.
The asserted function is $R\circ M \circ \Pi\colon \D \to \widehat{\C}$. This composition
is locally univalent, since each function itself is locally univalent. The composition
is surjective by the construction.


\section{Auxiliary results for Proposition~\ref{prop:normal} and Theorem~\ref{thm:carleson}}\label{sec:auxkey}

If $w$ is meromorphic in $\D$ and $S_w\in H^\infty_2$, then $w\in \ULU(\eta)$ for some sufficiently small
$\eta = \eta(\nm{S_w}_{H^\infty_2})$
by Nehari's theorem \cite[Corollary~6.4]{P:1982}.  Here $w\in \ULU(\eta)$ means that $w$ is
meromorphic and uniformly locally univalent in $\D$, or equivalently, there exists $0<\eta\leq 1$ such that $w$ is 
meromorphic and univalent
in each pseudo-hyperbolic disc $\Delta_p(a,\eta)$ for $a\in\D$.

Denote by $\mathcal{P}_w$ the (discrete) set of poles of the meromorphic function $w$ in $\D$.
Let $w\in \ULU(\eta)$ for some $0<\eta\leq 1$, and let $w'$ be a normal function such that $\mathcal{P}_w \neq \emptyset$.
By the Lipschitz-continuity of normal functions (as mappings from $\D$ equipped with the hyperbolic metric
to the Riemann sphere with the chordal metric), there exists a constant $s=s(\sigma(w'))$ with $0<s<1$ such that
\begin{equation} \label{eq:yama}
\abs{w'(z)} \geq 1, \quad z\in \Delta_p(a,s), \quad a\in \mathcal{P}_w;
\end{equation}
see for example \cite[Theorem~1]{Y:1986}.

The following lemma lists some elementary properties of uniformly locally univalent functions,
which are needed later. It is a local version of the well-known result according to which
$\log w'\in\BLOCH$ for all analytic and univalent functions $w$ in $\D$.


\begin{lemma} \label{lemma:ulu}
Let $w\in \ULU(\eta)$ for some $0<\eta\leq 1$, and let $0<s<1$ be fixed.
Suppose that $a\in\D$ is any point satisfying $\varrho_p( a,\mathcal{P}_w)\geq s$.
\begin{enumerate}
\item[\rm (i)]
Then
\begin{equation*} 
\left| \frac{w''(a)}{w'(a)} \right| (1-\abs{a}^2) \leq \frac{6}{\min\{\eta,s\}}.
\end{equation*}

\item[\rm (ii)]
For each $0<t<1$
there exists a constant $C=C(\eta,s,t)$ with $1<C<\infty$ such that
\begin{equation} \label{eq:myclaimpre}
C^{-1} \leq \left| \frac{w'(z_1)}{w'(z_2)}\right| \leq C, \quad z_1,z_2 \in \Delta_p\big(a,t \cdot \min\{\eta,s\}\big).
\end{equation}
\end{enumerate}
\end{lemma}


\begin{proof}
(i)
Let $\nu=\min\{\eta,s\}$. Since $g_a(z)=w(\varphi_a(\nu z))$ in univalent and analytic in~$\D$,
we deduce
\begin{equation} \label{eq:pomlog}
\sup_{z\in\D} \,\left| (1-|z|^2) \frac{g_a''(z)}{g_a'(z)} -2\overline{z} \right| \leq 4
\end{equation}
by \cite[Lemma~1.3]{P:1975}. Hence
\begin{equation*}
\left| \frac{w''(a)}{w'(a)} (1-\abs{a}^2) \nu - 2 \overline{a}\nu \right| = \left| \frac{g''_a(0)}{g'_a(0)} \right| \leq 4,
\end{equation*}
from which the assertion follows.

(ii) As above, let $\nu=\min\{\eta,s\}$.
Since $g_a(z) = w\big(\varphi_a(\nu z)\big)$ is analytic and univalent in $\D$, $\log g_a'$ 
is a well-defined analytic function whose Bloch-norm satisfies $\nm{ \log g_a'}_{\BLOCH}\leq 6$ by \eqref{eq:pomlog}.
For more information on Bloch functions we refer to \cite{ACP:1974}.
Recall that Bloch functions are precisely those analytic functions in $\D$ which are Lipschitz when the
unit disc is endowed with the hyperbolic metric and the plane with the Euclidean metric.
Hence, if $\zeta_1,\zeta_2\in D(0,t)$ for some fixed $0<t<1$, then \cite[Proposition~1, p.~43]{DS:2004} implies
\begin{equation}
\left| \log \left| \frac{g_a'(\zeta_1)}{g_a'(\zeta_2)} \right| \right|
       \leq \frac{\nm{ \log g_a'}_{\BLOCH}}{2} \, \log\frac{1+\varrho_p(\zeta_1,\zeta_2)}{1-\varrho_p(\zeta_1,\zeta_2)}
      \leq 3 \log\frac{1+\frac{2t}{1+t^2}}{1-\frac{2t}{1+t^2}}. \label{eq:defC}
\end{equation}
Let $K=K(t)$ be the constant defined by the right-hand side of \eqref{eq:defC}.
Consequently,
\begin{equation*}
e^{-K} \leq \left| \frac{g_a'(\zeta_1)}{g_a'(\zeta_2)} \right| \leq e^K, \quad \zeta_1,\zeta_2\in D(0,t),
\end{equation*}
from which  \eqref{eq:myclaimpre} follows for $C =  e^K (1+\nu)^2/(1-\nu)^2$.
\end{proof}

For $n\in\N$, the arcs 
\begin{equation*}
\big\{ e^{i\theta} \in \partial\D :  (j-1)2^{-n+2} \pi \leq \theta \leq j 2^{-n+2} \pi  \big\}, \quad j=1,\dotsc, 2^{n-1},
\end{equation*}
having pairwise disjoint interiors, constitute the $n^{\text{th}}$ generation of dyadic subintervals of $\partial\D$ --- the first generation being
$\partial\D$ itself. Analogously,
we may define dyadic subintervals of any arc $I\subset\partial\D$.

The set
\begin{equation*} 
Q  =Q_I = \big\{ z\in\D : 1-\abs{I}/(2\pi) \leq |z| < 1, \, \arg{z}\in I \big\}
\end{equation*}
is called a Carleson square,
where the interval $I\subset \partial\D$ is said to be the base of $Q$. The length of $Q$ 
is defined to be $\ell(Q)=\abs{I}$ (the arc-length of $I$),
while the top part (or the top half) of $Q$ is 
\begin{equation*}
T(Q) = \big\{ z\in Q : 1-\ell(Q)/(2\pi) \leq |z| \leq 1-\ell(Q)/(4\pi) \big\}.
\end{equation*}
Let $z_Q$ denote the center point of $T(Q)$.
Dyadic subsquares of a Carleson square~$Q$ are those Carleson squares whose bases are
dyadic subintervals of the base of~$Q$. Finally, a Carleson square $S$ is said to the father
of $Q$ provided that $Q$ is a dyadic subsquare of $S$ and $\ell(Q) = \ell(S)/2$. 

The following lemma is reminiscent of Lemma~\ref{lemma:ulu}(ii), and
hence its proof is omitted. The key point is that the top part of each Carleson square can be covered
by  finitely many pseudo-hyperbolic discs of fixed radius.


\begin{lemma} \label{lemma:carlesonest}
Let $w\in \ULU(\eta)$ for some $0<\eta\leq 1$, and let $w'$ be normal.
Suppose that $s=s(\sigma(w'))$ is a constant
such that \eqref{eq:yama} holds, and define $\lambda=(9/10+s)/(1+9s/10)$. Then, there exists
a~constant $C_0=C_0(\eta,\sigma(w'))$ with $1<C_0<\infty$ such that the following conclusions hold. 
\begin{enumerate}
\item[\rm (i)]
If $Q$ is a Carleson square such that 
$\varrho_p\big(T(Q),\mathcal{P}_w\big) \geq \lambda$,
and $S$ is the father-square of $Q$, then
\begin{equation*} 
C_0^{-1} \leq \left| \frac{w'(z_Q)}{w'(z_S)} \right| \leq C_0.
\end{equation*}
The same conclusion holds if $Q$ is the father-square of $S$.

\item[\rm (ii)]
If $Q$ is a Carleson square such that 
$\varrho_p\big(T(Q),\mathcal{P}_w\big) \geq \lambda$, then 
\begin{equation*} 
C_0^{-1} \leq \left| \frac{w'(z_1)}{w'(z_2)} \right| \leq C_0, \quad z_1,z_2\in T(Q).
\end{equation*}

\item[\rm (iii)]
If $Q$ is a Carleson square such that 
$\varrho_p\big(T(Q),\mathcal{P}_w\big) < \lambda$,  then 
\begin{equation*}
\abs{w'(z)} \geq C_0^{-1}, \quad z\in T(Q).
\end{equation*}
\end{enumerate}
\end{lemma}

The next result, which is based on an argument similar to that of \cite[Theorem~4]{BJ:1994}, 
allows us to control the number of those Carleson squares where $w'$ is small.
This information is crucial for our purposes since we want to prove that $1/w'$ belongs to some
Hardy space.


\begin{lemma} \label{lemma:L3}
Let $w$ be meromorphic in $\D$ such that $|S_w(z)|^2(1-|z|^2)^3\, dm(z)$ is a~Carleson measure,
and let $w'$ be normal.
Suppose that $C_0=C_0(\nm{S_w}_{F^2},\sigma(w'))$ with $1<C_0<\infty$ is the constant
ensured by Lemma~\ref{lemma:carlesonest}. 
Then, there exists a~constant $\varepsilon_0 = \varepsilon_0(\nm{S_w}_{F^2},\sigma(w'))$ with
$0<\varepsilon_0<\min\{1/4, C_0^{-1}\}$ having the following property:

If $Q$ is a Carleson square satisfying
$\abs{w'(z_Q)}\leq C_0^{-1/\varepsilon_0}$,
and $\{Q_j\}_{j=1}^\infty$ is the collection of maximal (with respect of inclusion) dyadic subsquares of $Q$ for which either
\begin{equation} \label{eq:assL3}
{\rm (i)}\,\,\,  \abs{w'(z_{Q_j})} \leq \varepsilon_0 \abs{w'(z_Q)} \qquad \text{or} \qquad 
{\rm (ii)} \,\,\,  \abs{w'(z_{Q_j})} \geq C_0^{-2},
\end{equation}
then $\sum_{j=1}^\infty \ell(Q_j) \leq \ell(Q)/2$.
\end{lemma}


\begin{proof}
Let $\RM = Q \setminus \bigcup_{j=1}^\infty Q_j$. By Lemma~\ref{lemma:carlesonest}(iii), 
$\RM$ is a simply connected subset of~$\D$, which does
not contain any poles of $w$ (nor poles of $w'$ for that matter). Even more is true, the pseudo-hyperbolic neighborhood 
of the radius $\lambda=\lambda(\sigma(w'))$ of~$\RM$
does not contain any poles of $w$, see Lemma~\ref{lemma:carlesonest} for the precise definition of $\lambda$.

The function $\log w'$ is analytic in $\RM$. By a standard limiting argument we may assume that $\RM$ is
compactly contained in $\D$.
We know that $\RM$ is a chord-arc domain\footnote{
If $\gamma\subset\C$ is a locally rectifiable closed curve, and there exists a constant $1\leq C<\infty$ such that the shorter arc connecting
any two points $z_1,z_2\in\gamma$ has arc-length at most $C\abs{z_1-z_2}$, then $\gamma$ is called
chord-arc. In particular, a domain in $\C$ is called chord-arc if its boundary is a~chord-arc curve.
Chord-arc curves are also known as Lavrentiev curves, and they are precisely the bi-Lipschitz images of circles.} 
\cite[p.~25]{B:2002} with some absolute chord-arc constant $1\leq C<\infty$.
Let $\Phi(z)=(z-z_Q)/\ell(Q)$, and denote $\mathcal{D} =\Phi(\RM)$. 
Now, $\mathcal{D}$ is a simply connected bounded chord-arc domain, which contains
the origin. Since $d(z_Q,\partial\RM)\asymp \diam(\RM) \asymp \ell(Q)$, we have 
$d(0,\partial\mathcal{D})\asymp \diam(\mathcal{\mathcal{D}})\asymp 1$, where
$d$ denotes the Euclidean distance while $\diam$ is the Euclidean diameter.

Define $F(\zeta) = \log w'\big(\Phi^{-1}(\zeta)\big) - \log w'\big(\Phi^{-1}(0)\big)$
for $\zeta\in\mathcal{D}$, and note that $F$ is analytic in $\mathcal{D}$.  We have
\begin{equation} 
  \int_{\mathcal{D}} \abs{F''(\zeta)}^2 \, d(\zeta,\partial\mathcal{D})^3 \, dm(\zeta) 
   = \frac{1}{\ell(Q)} \int_{\RM} \big| \big(\log{w'}\big)''(z)\big|^2 \, d(z,\partial\RM)^3 \, dm(z),\label{eq:lp}
\end{equation}
where $z=\Phi^{-1}(\zeta)$ and $d(\zeta,\partial\mathcal{D}) = d(z,\partial\RM)/\ell(Q)$.
By the identity $(\log w')'' = S_w + 2^{-1} (w''/w')^2$, the estimate $d(z,\partial\RM) \leq 1-\abs{z}^2$ for all $z\in\RM$,
and the assumption that $| S_w(z) |^2 \, (1-\abs{z}^2)^3 \, dm(z)$ is a~Carleson measure, \eqref{eq:lp} implies
\begin{align}
 & \int_{\mathcal{D}} \abs{F''(\zeta)}^2 \, d(\zeta,\partial\mathcal{D})^3 \, dm(\zeta)\notag\\
      & \qquad \lesssim  \frac{1}{\ell(Q)} \int_{\RM} | S_w(z) |^2 \, (1-\abs{z}^2)^3 \, dm(z)
      +  \frac{1}{\ell(Q)} \int_{\RM} \left| \frac{w''(z)}{w'(z)} \right|^4 \, d(z,\partial\RM)^3 \, dm(z) \notag\\ 
      & \qquad \lesssim \nm{S_w}^2_{F^2} +  \frac{1}{\ell(Q)} \int_{\RM} \left| \frac{w''(z)}{w'(z)} \right|^4 \, d(z,\partial\RM)^3 \, dm(z) \label{eq:uint}
\end{align}
with absolute comparison constants.

Our argument is based on the following auxiliary result, whose proof
is omitted. The proof of Lemma~\ref{lemma:objective} is a laborious but straightforward modification of the argument
in \cite[pp.~105-107]{BJ:1994}.


\begin{lemma} \label{lemma:objective}
Under the assumptions of Lemma~\ref{lemma:L3}:
For each $0<\varepsilon<\infty$
there exists a constant $C_1= C_1(\varepsilon,\nm{S_w}_{F^2},\sigma(w'))$ with $0<C_1<\infty$ such that
\begin{equation} \label{eq:objective}
\int_{\RM} \left| \frac{w''(z)}{w'(z)} \right|^4 \, d(z,\partial\RM)^3 \, dm(z)
 \leq C_1 \, \ell(Q) + 
\varepsilon^2 \, \int_{\RM} \left| \frac{w''(z)}{w'(z)} \right|^2 \, d(z,\partial\RM) \, dm(z).
\end{equation}
\end{lemma}

We continue with the proof of Lemma~\ref{lemma:L3}.
By combining \eqref{eq:uint} and \eqref{eq:objective}, the change
of variable gives
\begin{equation} \label{eq:pre}
  \int_{\mathcal{D}} \abs{F''(\zeta)}^2 \, d(\zeta,\partial\mathcal{D})^3 \, dm(\zeta)
  \lesssim \nm{S_w}^2_{F^2} + C_1 + \varepsilon^2
   \int_{\mathcal{D}} \left| F'(\zeta) \right|^2 \, d(\zeta,\partial\mathcal{D}) \, dm(\zeta)
\end{equation}
with absolute comparison constants.
We apply a well-known version of Green's formula \cite[Lemma~3.6]{BJ:1994} for the domain $\mathcal{D}$.
Since 
$F(0)=0$ and $\ell(Q) \leq (4/3)(1-\abs{z_Q}^2)$, Lemma~\ref{lemma:ulu}(i) and \eqref{eq:pre} imply
\begin{align*}
\int_{\partial\mathcal{D}} \abs{F(\zeta)}^2 \, \abs{d\zeta}
   & \lesssim  \abs{F'(0)}^2 +   \int_{\mathcal{D}} \abs{F''(\zeta)}^2 \, d(\zeta,\partial\mathcal{D})^3 \, dm(\zeta)\\
   & \lesssim \left| \frac{w''(z_Q)}{w'(z_Q)} \right|^2 \ell(Q)^2 + \nm{S_w}^2_{F^2} + C_1
  + \varepsilon^2 \int_{\mathcal{D}} \left| F'(\zeta) \right|^2 \, d(\zeta,\partial\mathcal{D}) \, dm(\zeta)\\
    & \lesssim  \left( \frac{1}{\min\{\eta,s\}} \right)^2 + \nm{S_w}^2_{F^2} + C_1
+ \varepsilon^2 \int_{\partial\mathcal{D}} \abs{F(\zeta)}^2 \, \abs{d\zeta}
\end{align*}
with absolute comparison constants.
If $0<\varepsilon<\infty$ is sufficiently small, then the computation above shows that there 
exists a constant $C_2=C_2(\nm{S_w}_{F^2}, \sigma(w'))$ with $0<C_2<\infty$  such that
\begin{equation} \label{eq:kk}
\int_{\partial\RM} \big| \log w'(z) - \log w'(z_Q)\big|^2 \, \abs{dz}
    = \ell(Q) \, \int_{\partial\mathcal{D}} \abs{F(\zeta)}^2 \, \abs{d\zeta} \leq C_2 \, \ell(Q).
\end{equation}

Let $T_j$ denote the top of $\partial Q_j$ (i.e.~the roof of $Q_j$) for $j\in\N$. 
Since the roofs $T_j \subset \partial\RM$ are pairwise disjoint, 
we deduce from \eqref{eq:kk} that
\begin{equation} \label{eq:lengths}
\sum_{j=1}^\infty \, \int_{T_j} \big| \log w'(z) - \log w'(z_Q) \big|^2 \, \abs{dz} \leq C_2 \, \ell(Q).
\end{equation}
There are two types of subsquares $Q_j$, which result from 
\eqref{eq:assL3}. 
\begin{itemize}
\item[\rm (i)]
In the case of type (i) squares, Lemma~\ref{lemma:carlesonest}(iii)
shows that $\varrho_p(T(Q_j),\mathcal{P}_w)\geq \lambda$. Hence
Lemma~\ref{lemma:carlesonest}(ii) implies
\begin{equation*}
\abs{w'(z)} \leq C_0 \abs{w'(z_{Q_j})} \leq C_0 \, \varepsilon_0 \abs{w'(z_Q)}, \quad z\in T_j,
\end{equation*}
and further,
\begin{equation} \label{eq:log1x}
\begin{split}
\big| \log w'(z) - \log w'(z_Q) \big| 
         & \geq  \log{\abs{w'(z_Q)}} - \log{\abs{w'(z)}}\\
       & \geq \log{(C_0 \varepsilon_0)^{-1}} >0, \quad z\in T_j.
\end{split}
\end{equation}
\item[\rm (ii)]
In the case of type (ii) squares $Q_j$, let $S_j$ be their father-squares, respectively. Now,
by Lemma~\ref{lemma:carlesonest} 
\begin{equation*}
\abs{w'(z)} \geq C_0^{-1} \abs{w'(z_{S_j})} \geq C_0^{-2} \abs{w'(z_{Q_j})} \geq C_0^{-4}, \quad z\in T_j.
\end{equation*}
Since $\abs{w'(z_Q)} \leq C_0^{-1/\varepsilon_0}$, and $\varepsilon_0<1/4$, we get
\begin{equation} \label{eq:log2}
\begin{split}
  \big| \log w'(z) - \log w'(z_Q) \big| 
  & \geq \log \abs{w'(z)} - \log \abs{w'(z_Q)}\\
  &  = (\varepsilon_0^{-1} - 4) \log C_0>0, \quad z\in T_j.
\end{split}
\end{equation}
\end{itemize}

By combining \eqref{eq:lengths}, \eqref{eq:log1x} and \eqref{eq:log2},
and by choosing $\varepsilon_0=\varepsilon_0(\nm{S_w}_{F^2}, \sigma(w'))$ with $0<\varepsilon_0<C_0^{-1}$ sufficiently small, 
we conclude $\sum_{j=1}^\infty \ell(Q_j) \leq 2^{-1} \, \ell(Q)$.
The assertion of Lemma~\ref{lemma:L3} follows. 
\end{proof}


\section{Proof of Proposition~\ref{prop:normal}}

Let $f$ be a solution of \eqref{eq:de2} with $A\in H^\infty_2$, and let
$\{z_n\}_{n=1}^\infty$ be the zero-sequence of $f$.
Implication ${\rm (i)} \Rightarrow {\rm (ii)}$ is a direct consequence of \eqref{eq:normaldef}, and
hence it suffices to prove ${\rm (ii)} \Rightarrow {\rm (iii)}$ and ${\rm (iii)} \Rightarrow {\rm (i)}$.

${\rm (ii)} \Rightarrow {\rm (iii)}$: Denote $K = \sup_{n\in\N} \, (1-|z_n|^2) |f'(z_n)| < \infty$. Fix $n\in\N$,
and let $z\in \overline{D}(z_n,c(1-|z_n|))$ be any point satisfying
\begin{equation*}
\max_{\zeta\in \overline{D}(z_n,c(1-|z_n|))} |f(\zeta)| = |f(z)|.
\end{equation*}
Here $\overline{D}(z_n,c(1-|z_n|))$ denotes a closed Euclidean disc centered at $z_n\in\D$, and $c$ is a sufficiently
small constant (to be defined later). By means of \eqref{eq:de2} we deduce
\begin{align*}
|f(z)| & 
       = \left| \int_{z_n}^z  \left( f'(z_n) + \int_{z_n}^\zeta f''(s) \, ds \right) d\zeta \right|\\
       & \leq c(1-|z_n|) |f'(z_n)| + |f(z)| \, \int_{z_n}^z  \int_{z_n}^\zeta \frac{\nm{A}_{H^\infty_2}}{(1-|s|)^2} \, |ds|  |d\zeta|\\
       & \leq c K + |f(z)| \, \frac{c^2(1-|z_n|)^2\nm{A}_{H^\infty_2}}{(1-c)^2(1-|z_n|)^2}.
\end{align*}
If $c=c(\nm{A}_{H^\infty_2})$ is sufficiently small to satisfy $c^2 \nm{A}_{H^\infty_2}/(1-c)^2<1$, then we get
\begin{equation*}
|f(\zeta)| \leq \frac{cK}{1 - c^2 \nm{A}_{H^\infty_2}/(1-c)^2}, \quad \zeta \in \bigcup_{n=1}^\infty D\big(z_n,c(1-|z_n|) \big).
\end{equation*}

${\rm (iii)} \Rightarrow {\rm (i)}$: Suppose that $f$ satisfies
\begin{equation} \label{eq:kkass}
|f(z)| \leq C, \quad z\in \bigcup_{n=1}^\infty D\big( z_n,c(1-|z_n|) \big),
\end{equation}
for some constants $0<c<1$ and $0<C<\infty$. Let $g$ be a linearly independent solution to $f$ such
that the Wronskian determinant $W(f,g) = 1$. Define $w=g/f$, and notice that $S_w = 2A$ and $w'=1/f^2$.
Now $S_w\in H^\infty_2$ implies $w\in \ULU(\eta)$ for any sufficiently small
$\eta = \eta(\nm{S_w}_{H^\infty_2})$ by Nehari's theorem \cite[Corollary~6.4]{P:1982}.
Let $t=t(\nm{S_w}_{H^\infty_2},c)$ with $0<t<1$ be a sufficiently small 
constant such that 
\begin{enumerate}
\item[\rm (a)]
$\overline{\Delta}_p(z_n,\eta t)\subset D(z_n,c(1-|z_n|))$ for all $n\in\N$; 
\item[\rm (b)]
$2\eta t/(1+\eta^2 t^2)<\eta$.
\end{enumerate}
Here $\overline{\Delta}_p(a,\eta t)$ stands for the closed pseudo-hyperbolic disc of radius $\eta t$, centered at $a\in\D$.
We proceed to verify \eqref{eq:normaldef} in two parts.

First, suppose that $a\in\D$ is a point such that $\varrho_p(a,\mathcal{P}_w)\geq  \eta t$. By Lemma~\ref{lemma:ulu}(i)
we deduce
\begin{equation*}
(1-|a|^2) \, \frac{|f'(a)|}{1+|f(a)|^2}  
 \leq \frac{1}{4} \,  (1-|a|^2) \, \frac{|w''(a)|}{|w'(a)|} \leq \frac{3}{2\eta t}.
\end{equation*}

Second, suppose that $a\in\D$ is a point such that
$\varrho_p(a,\mathcal{P}_w)<  \eta t$, or equivalently, $a\in\Delta_p(z_n,\eta t)$
for some $n\in\N$.
By the maximum modulus principle 
there exists a~point $s_n \in \partial\Delta_p(z_n,\eta t)$ such 
that
\begin{equation*}
\max_{z\in\overline{\Delta}_p(z_n,\eta t)} \abs{f'(z)} = \abs{f'(s_n)}.
\end{equation*}
Note that $\Delta_p(s_n,\eta t)$ does not
contain any zeros of $f$ (any such zero would lie too close to $z_n$ by the condition (b) and the fact $w\in\ULU(\eta)$).
Lemma~\ref{lemma:ulu}(i) yields
\begin{equation} \label{eq:maxmod}
(1-\abs{s_n}^2) \abs{f'(s_n)} \leq \frac{3}{\eta t} \, \abs{f(s_n)}.
\end{equation}
Since $a,s_n\in\overline{\Delta}_p(z_n,\eta t)$, there exists a constant $K=K(\nm{S_w}_{H^\infty_2},c)$
with $1<K<\infty$ such that
$1/K \leq (1-\abs{a}^2)/(1-\abs{s_n}^2) \leq K$.
By means of the maximum modulus principle, \eqref{eq:kkass} and \eqref{eq:maxmod} we deduce
\begin{equation*}
 (1-\abs{a}^2) \frac{\abs{f'(a)}}{1+\abs{f(a)}^2}
   \leq \frac{1-|a|^2}{1-|s_n|^2} \Big(  (1-\abs{s_n}^2) \abs{f'(s_n)} \Big)  \frac{1}{1+\abs{f(a)}^2}
   \leq \frac{3CK}{\eta t}.
\end{equation*}

We have proved $\sigma(f)<\infty$, and hence $f$ is normal. This concludes the proof of Proposition~\ref{prop:normal}.


\section{Proof of Theorem~\ref{thm:carleson}}

We proceed to show that the non-tangential maximal function
\begin{equation*} 
(1/w')^\star(e^{i\theta}) = \sup_{z\in\Gamma_\alpha(e^{i\theta})} \, \frac{1}{\abs{w'(z)}}, \quad
e^{i\theta}\in\partial\D,
\end{equation*}
belongs to the weak Lebesgue space $L^p_w(\partial\D)$ for some $0<p<\infty$,
which is to say that there exists a constant $C=C(\alpha,w')$ with $0<C<\infty$ such that
the distribution function satisfies
\begin{equation*} 
\big| \big\{ e^{i\theta} \in\partial\D : (1/w')^{\star}(e^{i\theta})  > \lambda \big\} \big|
     \leq \frac{C}{\lambda^p}, \quad 0<\lambda<\infty.
\end{equation*}
This leads to the assertion, since $L^p_w(\partial\D) \subset L^q(\partial\D)$ for any $0<q<p$.
Here $\Gamma_\alpha(e^{i\theta}) = \{z\in\D : \abs{z-e^{i\theta}} \leq \alpha(1-\abs{z})\}$, for fixed $1<\alpha<\infty$,
is a non-tangential approach region with vertex at $e^{i\theta}\in\partial\D$ with aperture of $2\arctan\sqrt{\alpha^2-1}$,
and the absolute value of the set is 
its one dimensional Lebesgue measure. 

Let $C_0 = C_0(\nm{S_w}_{F^2}, \sigma(w'))$ with $1<C_0 < \infty$ 
be the constant ensured by Lemma~\ref{lemma:carlesonest},
and $\varepsilon_0 = \varepsilon_0(\nm{S_w}_{F^2}, \sigma(w'))$ with $0<\varepsilon_0<\min\{ 1/4,C_0^{-1}\}$
be the constant resulting from Lemma~\ref{lemma:L3}.
We consider the collection $\mathcal{G}_0$ of maximal (with respect to inclusion) dyadic subsquares
of $\D$, of at least second generation, satisfying
\begin{equation*} 
|w'(z_{Q})|\leq C_0^{-1/\varepsilon_0}, \quad Q\in\mathcal{G}_0.
\end{equation*}
Denote $L=L(C_0,\varepsilon_0) = C_0^{1+1/\varepsilon_0}$ and $M=M(C_0,\varepsilon_0) = C_0/\varepsilon_0$,
for short. The maximality and Lemma~\ref{lemma:carlesonest} imply 
\begin{equation*} 
\abs{w'(z)}>L^{-1}, \quad z\in\D \setminus \bigcup_{Q\in\mathcal{G}_0} Q,
\end{equation*}
and $|w'(z_{Q})|> L^{-1}$ for all $Q\in\mathcal{G}_0$.

Apply Lemma~\ref{lemma:L3} to each $Q\in\mathcal{G}_0$ to get the collection $\mathcal{S}_Q$ of 
dyadic subsquares of~$Q$, where the subsquares are maximal with the property
\begin{equation*}
\abs{w'(z_S)} \leq \varepsilon_0 \, \abs{w'(z_{Q})},
\quad S\in\mathcal{S}_Q.
\end{equation*}
Note that the squares in the family $\mathcal{S}_Q$ are either the ones appearing in (i) of Lemma~\ref{lemma:L3} or
are contained in the ones appearing in (ii).
Let $\mathcal{G}_1 = \bigcup_{Q\in\mathcal{G}_0} \mathcal{S}_Q$.
For any $Q\in\mathcal{G}_0$,
the maximality and Lemma~\ref{lemma:carlesonest} yield 
\begin{equation*}
\abs{w'(z)}
> \varepsilon_0 \, \abs{w'(z_{Q})} \, C_0^{-1}
> L^{-1} M^{-1}, \quad
z\in\D \setminus \bigcup_{Q^{\star}\in\mathcal{G}_1} Q^{\star};
\end{equation*}
$|w'(z_S)|
> \varepsilon_0 \, \abs{w'(z_{Q})} \, C_0^{-1}
>L^{-1}M^{-1}$  for all $S\in\mathcal{S}_Q$;
and moreover, we have $\sum_{S\in\mathcal{S}_Q} \ell(S) \leq \ell(Q)/2$ by Lemma~\ref{lemma:L3}.
We deduce
\begin{equation*}
\sum_{Q\in\mathcal{G}_1} \ell(Q) \leq \frac{1}{2} \, \sum_{Q\in\mathcal{G}_0} \ell(Q).
\end{equation*}
Repeating this process inductively, we obtain collections $\mathcal{G}_n$ for $n\in\N$
such that
\begin{equation*}
\abs{w'(z)}> L^{-1} M^{-n}, \quad
z\in\D \setminus \bigcup_{Q\in\mathcal{G}_n}^\infty Q, \quad n\in\N,
\end{equation*}
and
\begin{equation*}
\sum_{Q\in \mathcal{G}_n} \ell(Q) \leq \frac{1}{2^n} \sum_{Q\in\mathcal{G}_0} \ell(Q), \quad n\in\N.
\end{equation*}

Fix any $\lambda$ for which $L M \leq \lambda <\infty$,
and choose the natural number $N$ such that $L M^N \leq \lambda < L M^{N+1}$. Now
\begin{align*}
\big\{ e^{i\theta} \in\partial\D : ( 1/w' )^{\star} (e^{i\theta})  > \lambda \big\}
 & = \Big\{ e^{i\theta} \in\partial\D : \inf_{z\in\Gamma_\alpha(e^{i\theta})} \abs{w'(z)} < 1/\lambda \Big\}\\
& \subset \Big\{ e^{i\theta} \in\partial\D : \inf_{z\in\Gamma_\alpha(e^{i\theta})} |w'(z)|  < L^{-1} M^{-N}  \Big\}.
\end{align*}
By the inductive process above, there exists a constant $K=K(\alpha)$ with $1<K<\infty$  such that
\begin{equation*}
\big| \big\{ e^{i\theta} \in\partial\D : \left( 1/w' \right)^{\star} (e^{i\theta})  > \lambda \big\} \big|
        \leq \frac{K}{2^N} \sum_{Q\in\mathcal{G}_0} \ell(Q)
         \leq   \frac{4\pi K L^{(\log_2 M)^{-1}}}{\lambda^{(\log_2 M)^{-1}}}.
\end{equation*}
Hence $(1/w')^\star\in L^p_w(\partial\D)$  for $p= 1/( \log_2 C_0/\varepsilon_0 )$,
and consequently, $1/w' \in H^q$ for any $0<q<p$. This proves the assertion of Theorem~\ref{thm:carleson}.


\section{Proof of Corollary~\ref{cor:34}}

Let $f$ be a normal non-trivial solution of \eqref{eq:de2}, where
$|A(z)|^2(1-|z|^2)^3\, dm(z)$ is a~Carleson measure. Let $g$ be
a solution of \eqref{eq:de2} such that $W(f,g)=1$.
Now $w=g/f$ satisfies $S_w = 2 A$.
Since $w' = 1/f^2$ is normal ($\sigma(w') \leq 2 \,\sigma(f) < \infty$),
Theorem~\ref{thm:carleson} asserts that 
for any sufficiently small $0<p<\infty$
there exists a constant $C=C(p,\nm{A}_{F^2}, \sigma(f))$ with $1<C<\infty$ such that
$\nm{f}_{H^{2p}}^{2p} = \nm{1/w'}^p_{H^p} \leq C$.


\end{document}